\documentclass[11pt,a4paper]{article}
\usepackage[english]{babel}
\usepackage{german}
\usepackage{amsmath,amsthm,amssymb,amsfonts}
\usepackage[utf8x]{inputenc}
\usepackage{hyperref}
\usepackage{fontenc}
\usepackage{enumerate}
\usepackage{a4wide}
\usepackage[hang]{footmisc}
\usepackage{delarray}
\usepackage{fancybox}
\usepackage{graphicx}
\usepackage[T1]{fontenc}
\usepackage{amscd}
\usepackage[nohug]{diagrams}
\usepackage{epigraph}
\usepackage{textcomp}
\diagramstyle[labelstyle=\scriptstyle]

\newcommand{\F}{$\mathcal{F}$}
\newcommand{\N}{\mathbb{N}}

\newcommand{\m}{\mathfrak{m}}

\newtheorem{theorem}{Theorem}[section]
\newtheorem{corollary}[theorem]{Corollary}
\newtheorem{definition}[theorem]{Definition}
\newtheorem{lemma}[theorem]{Lemma}
\newtheorem{remark}[theorem]{Remark}
\newtheorem{example}[theorem]{Example}
\newtheorem{conjecture}[theorem]{Conjecture}

\newtheorem*{theorem*}{Theorem}

\begin{document}

\title{\textbf{An application of generalized Matlis duality for quasi-\F-modules to the Artinianness of local cohomology modules}}

\author{Danny Tobisch \\ \\ Universität Leipzig, Mathematisches Institut\\Johannisgasse 26, 04103 Leipzig, Germany\\E-Mail: \href{mailto:tobisch@math.uni-leipzig.de}{tobisch@math.uni-leipzig.de} }

\maketitle

\abstract{\noindent We use a result of Hellus about generalized local duality to describe some generalized Matlis duals for certain quasi-\F-modules.
Furthermore, we apply this description to obtain examples for non-artinian local cohomology modules by the theory of \F-modules. 
In particular, we get a new view on Hartshorne's counterexample for a conjecture by Grothendieck about the finiteness of $Hom_R\left(R/I,H^i_I\left(R\right)\right)$ for a noetherian local Ring $R$ and an ideal $I \subseteq R$.


\section{Introduction}

In 1992, Huneke \cite{Hu92} stated four basic problems about local cohomology. One of these is the question whether or not a given local cohomology module is artinian. \\ 
Let $\left(R,\m, \Bbbk\right)$ be a noetherian local ring and $M$ a finitely generated $R$-module. Then it is well-known that the local cohomology module $H_{\m}^i\left(M\right)$ with support in the 
maximal ideal $\m$ is artinian for all $i$. On the other hand this is equivalent to both of the statements $Supp_R\left(M\right)\subseteq \{\m\}$ and the fact that $Hom_R\left(R/\m,M\right)$ is finitely generated. 
Regarding this, Grothendieck conjectured the following.

\begin{conjecture}[Exposé XIII/Conjecture 1.1 in \cite{Gro68}]\label{VermGroth}
Let $\left(R,\m\right)$ be a noetherian local ring, $I\subseteq R$ an ideal and $M$ a finitely generated $R$-module, then
$Hom_R\left(R/I, H^i_I\left(M\right)\right)$ is finitely generated for all $n \in \N$.
\end{conjecture}

\noindent
But in \cite{Ha70} Hartshorne showed this to be false even for regular rings $R$ by giving a counterexample. He showed that for the ring $R[[u,v,x,y]]$ and the ideals $\mathfrak{a}=\left(ux+vy\right)R$ 
and $I=\left(u,v\right)R$ the module $Hom_R\left(R/\m, H^2_I\left(R/\mathfrak{a}\right)\right)$ is not finitely generated, and \\
$Hom_R\left(R/I, H^2_I\left(R/\mathfrak{a}\right)\right)$ cannot be finitely generated either. 
In particular, $H_I^2\left(R/\mathfrak{a}\right)$ is not artinian. \\
This example was generalized by Stückrad and Hellus in \cite{HS09} to all modules of the form $R/p$ with a prime $p \in \left(x,y\right)$. They used the fact that certain Matlis duals have infinitely 
many associated primes. In fact, they proved the following theorem.

\begin{theorem}\label{GeneralizedExample}
Let $\Bbbk$ be any field, $R=\Bbbk[[X_1,\dots ,X_n]]$ the ring of formal power series in the variables $X_1,\dots ,X_n$ ($n\geq4$) and $\mathfrak{a}$ the ideal $\left(X_1,\dots ,X_{n-2}\right)R$. 
If $p\in R$ is prime with $p\in \left(X_{n-1},X_n\right)R$, then 
\begin{equation*}
 H_{\mathfrak{a}}^{n-2}\left(R/pR\right)
\end{equation*}
is not artinian.
\end{theorem}
\begin{proof}
 \cite[Theorem 2.4]{HS09}
\end{proof}

\noindent
In this paper we only consider the case of prime characteristic. But so we can find some new relations between the theory of \F-modules which was firstly introduced by Lyubeznik in 
\cite{Lyu97} and Hartshorne`s Example. We will show that we can translate the question whether certain local cohomology modules over a ring of formal power series 
are artinian into the task to decide whether a given \F-module is \F-finite. 
\\

\noindent
More precisely we will proove:

\hspace{1pt}
\begin{theorem*}[Theorem \ref{generalizedMatliswithHellus2}]
Let $\left(R,\m,\Bbbk\right)$ be a complete regular local ring of characteristic $p>0$ with perfect residue field $\Bbbk$ and let $I,\mathfrak{a}\subseteq R$ be ideals of $R$.
If furthermore $R/\mathfrak{a}$ is Cohen-Macaulay, we have for all $i\in \{0,\dots,height\,\mathfrak{a}\}$ 
\begin{equation*}
\mathfrak{D}\left(H_\mathfrak{a}^{height\,\mathfrak{a}-i}\left(R/I\right)\right)\cong H^i_I\left(D\left(H_\mathfrak{a}^{height\,\mathfrak{a}}\left(R\right)\right)\right).
\end{equation*}
\end{theorem*}

\hspace{1pt}

\noindent
We will see that this $R$-module is an \F-module, but if it is not \F-finite the local cohomology module $H_\mathfrak{a}^{height\,\mathfrak{a}-i}\left(R/I\right)$ cannot be artinian.
\\

\noindent
The main ingredients for this theorem are a generalized local duality, formulated by Hellus in his habilitation thesis (\cite{HE07a}) and an extension of usual Matlis duality to the 
category of quasi-\F-modules, which was firstly formulated by Lyubeznik in \cite{Lyu97} for cofinite modules and later generalized by Blickle in \cite{BL01}. 

\section{Generalized local duality}\label{section2}

For a complete local ring $\left(R,\m\right)$ Matlis duality provides a correspondence between the category of noetherian $R$-modules and the category of artinian $R$-modules. So it is a quite 
interesting question which finitely generated modules correspond to the artinian local cohomology module $H_{\m}^i\left(M\right)$ for a finitely generated $R$-module $M$. 
The local duality theorem answers this question.

\begin{theorem}[local duality]
 Let $\left(R,\m,\Bbbk\right)$ be a local d-dimensional Cohen-Macaulay ring with canonical module $\omega_R$. Let $M$ be a finitely generated $R$-module. Then we have for all $0\leq i \leq d$:
\begin{equation*}
 H^i_\mathfrak{m}\left(M\right) \cong D_R\left(Ext_R^{d-i}\left(M,\omega_R\right)\right).
\end{equation*}
\end{theorem}
\begin{proof}
 \cite[Theorem 11.44]{IY07}.
\end{proof}

\noindent
In the above version of local duality, we have to consider for a local ring $\left(R,\m\right)$ the local cohomology modules with support in $\m$. 
In \cite{HE07a} Hellus could generalize this to a wider class of support-ideals under certain assumptions.

\begin{theorem}[generalized local duality]\label{generalizedMatlis-duality}
Let $\left(R,\m\right)$ be a noetherian local ring, $I\subseteq R$ an ideal and $h\in \N$, such that
\begin{equation*}
 H_I^l\left(R\right)\neq0 \Longleftrightarrow l=h,
\end{equation*}
and let $M$ be an $R$-module.
then for all $i\in\{0,\dots,h\}$ we have a natural isomorphism
\begin{equation*}
D\left(H_I^{h-i}\left(M\right)\right) \cong Ext_R^i\left(M,D\left(H_I^h\left(R\right)\right)\right).
\end{equation*}
\end{theorem}
\begin{proof}
 \cite[Theorem 6.4.1]{HE07a}.
\end{proof}

\noindent
The next corollary shows that this is just a generalization of the usual local duality.

\begin{corollary}
 Let $\left(R,\m\right)$ be a noetherian complete local Cohen-Macaulay ring and let $M$ be a finitely generated $R$-module. Additionally let be $\omega_R:=D\left(H_\m^{dim\left(R\right)}\left(R\right)\right)$. 
Then we have an isomorphism
\begin{equation*}
  H_\m^{dim\left(R\right)-i}\left(M\right) \cong D\left(Ext_R^i\left(M,\omega_R\right)\right).
\end{equation*}
\end{corollary}
\begin{proof}
\cite[Remark 6.4.2]{HE07a}
\end{proof}

\section{Generalized Matlis-duality for quasi-\F-modules}

\noindent
Usual Matlis duality over a complete local ring $\left(R,\m\right)$ allows facts about artinian $R$-modules to be translated into corresponding statements about noetherian $R$-modules. 
Motivated by Lyubeznik`s functor $\mathcal{H}_{R,A}$ for cofinite $R\{f\}$-modules, presented in his influental paper \cite{Lyu97}, Blickle was able to extend the usual Matlis duality, defined as
$D\left(M\right):=Hom_R\left(M,E\left(R/\m\right)\right)$, in \cite{BL01} to the category of quasi-\F-modules. He showed that this duality functor extends to a functor 
$D_\mathcal{F}$: quasi-\F-mod $\rightarrow$ quasi-\F-mod, which involves Frobenius structures. We follow Blickle in \cite{BL01}, who constructed this functor to analyze some duality 
properties for quasi-\F-modules.
In the case of a complete regular local ring $\left(R,\m,\Bbbk\right)$ with perfect residue field we get in a functor
\begin{equation*}
\mathfrak{D} : \mathcal{QF}\textrm{-mod}  \longrightarrow  \mathcal{F}\textrm{-mod} .
\end{equation*}

\noindent
Let $R$ be a noetherian commutative ring of positive characteristic $p$. The Frobenius homomorphism $\varphi:\, R \rightarrow  R,\,r \mapsto r^p $ provides $R$ with a nontrivial $R$-bimodule 
structure, given by the usual left action and the right action by the Frobenius. Therewith the Frobenius functor $\mathcal{F}:$ $R$-mod $\rightarrow$ $R$-mod is defined due to Peskine and Szpiro 
in \cite{PS73}:
 \begin{equation*}
\mathcal{F}\left(M\right) = R^\varphi \otimes_R M
\end{equation*}
\begin{equation*}
\mathcal{F}\left(M \xrightarrow{f} N\right) = \left(R^\varphi \otimes_R M \xrightarrow{id\otimes_R f} R^\varphi \otimes_R N\right).
\end{equation*}

\noindent
Now we are able to introduce the notation of a quasi-\F-module which is inspired by the $R[F]$-modules in \cite{BL01}, resp. by the $R\{f\}$-modules in \cite{Lyu97}.
By avoiding these notions we particulary want to emphasise the relation to Lyubeznik`s \F-modules. Nonetheless all these definitions are equivalent (see e.g. \cite[section 2.2]{BL01}).

\begin{definition}[quasi-\F-module]\label{DefQuasiF}
A \textbf{\textit{quasi-\F-module}} is a pair $\left(M, \beta\right)$, consisting of an $R$-module $M$ and a $R$-linear map
\begin{equation*}
 \beta: \mathcal{F}\left(M\right)=R^\varphi \otimes_R M \rightarrow M,
\end{equation*}
which we call \textbf{\textit{structure morphismus}} of $M$.
A morphism between two quasi-\F-modules $\left(M, \beta\right)$ and $\left(M',\beta'\right)$ is an $R$-module homomorphism $f:\, M \rightarrow M'$, such that the following diagram commutes
\begin{equation*}
 \begin{CD}
M @>f>> M'\\
@VV\beta V @VV\beta'V\\
\mathcal{F}\left(M\right) @>\mathcal{F}(f)>> M'.
\end{CD}
\end{equation*}
A quasi-\F-module $\left(M, \beta\right)$ is called \textbf{\textit{\F-module}} iff $\beta$ is an isomorphism and we call an \F-module $\left(M, \beta\right)$ \textbf{\textit{\F-finite}} iff we could obtain 
the module $M$ by a direct limit process of the form 
\begin{equation*}
 M= \varinjlim \left(N\xrightarrow{\theta}\mathcal{F}\left(N\right)\xrightarrow{\mathcal{F}\left(\theta\right)}\mathcal{F}^2\left(N\right)\xrightarrow{\mathcal{F}^2\left(\theta\right)}\mathcal{F}^3\left(N\right)\xrightarrow{\mathcal{F}^3\left(\theta\right)}\dots\right)
\end{equation*}
with $N \in R$-mod finitely generated.
\end{definition}

\noindent
Let $\left(R,\m,\Bbbk\right)$ be a complete regular local ring of characteristic $p>0$ and let $\left(M,\beta\right)$ be a quasi-\F-module. In this situation, by Cohen`s structure theorem 
(\cite[Theorem 8.28]{IY07}), $R$ is isomorphic to a ring of formal power series in finitely many variables over the field $\Bbbk$. If, on top of this, $\Bbbk$ is perfect, i.e. 
$\Bbbk^p=\Bbbk$, $R$ is finitely generated over $R^p$. This means that $R$ is a so-called $F$-finite ring and, by \cite[corollary 4.10]{BL01}, we get a natuaral isomorphism 
$\mathcal{F}\left(Hom_R\left(M,N\right)\right)\cong Hom_R\left(\mathcal{F}\left(M\right),\mathcal{F}\left(N\right)\right)$ for all $R$-modules $M$ and $N$. Since the injective hull of the residue field $E\left(R/\m\right)$ is in fact an \F-module, 
i.e. $\mathcal{F}\left(E\left(R/\m\right)\right)\cong E\left(R/\m\right)$, we get an isomorphism
\begin{equation*}
\tau_M: \, \mathcal{F}\left(Hom_R\left(M,E\left(R/\m\right)\right)\right) = \mathcal{F}\left(D\left(M\right)\right) \cong D\left(\mathcal{F}\left(M\right)\right) = Hom_R\left(\mathcal{F}\left(M\right),E\left(R/\m\right)\right), 
\end{equation*}
for all $R$-modules $M$. Matlis duality yields a map
\begin{equation*}
 \gamma : \, D\left(M\right) \xrightarrow{D\left(\beta\right)} D\left(\mathcal{F}\left(M\right)\right) \xrightarrow{\tau_M} \mathcal{F}\left(D\left(M\right)\right).
\end{equation*}

\noindent
With this map, Blickle defined the following functor (see \cite[section 4.2]{BL01} for details).

\begin{definition}\label{defgeneralizedMatlis}
Let $\left(R,\m,\Bbbk\right)$ be a complete regular local ring of positive characteristic $p>0$ and let $\left(M,\beta\right)$ be a quasi-\F-module 
(finitely generated or artinian as $R$-module if $\Bbbk$ is not perfect). Let $\gamma:=\tau_M \circ D\left(\beta\right)$. Then 
\begin{equation*}
\mathfrak{D}\left(M\right):=\varinjlim\, \left(D\left(M\right) \xrightarrow{\gamma} \mathcal{F}\left(D\left(M\right)\right) \xrightarrow{\mathcal{F}\left(\gamma\right)} \mathcal{F}^2\left(D\left(M\right)\right)
\xrightarrow{} \dots \,\right)
\end{equation*} 
is an \F-module generated by $\gamma$. On the above-mentioned class of modules (resp. rings) this construction defines an exact functor.
\end{definition}

\noindent
The exactness is obvious by the exactness of the usual Matlis duality functor and the direct limit. If, even more, $M$ is an \F-module, hence $\gamma$ an isomorphism, 
the direct system only consists of one element and we have $\mathfrak{D}\left(M\right)=D\left(M\right)$.
If $M$ is artinian, by Matlis duality $D\left(M\right)$ is finitely generated and $\mathfrak{D}\left(M\right)$ is in fact \F-finite. 

\begin{remark}\label{remark1}
If the quasi-\F-module $\left(M,\beta\right)$ is an \F-module, i.e. $\mathcal{F}\left(M\right)\cong M$, over the complete regular local ring $\left(R,\m,\Bbbk\right)$ with perfect residue field $\Bbbk$, we see that 
the map $\tau_M$ from above yields an \F-module structure on the Matlis dual $D(M):=Hom_R(M,E(R/\m))$. In fact by precomposing with $\beta^{-1}$ we get an isopmorphism
\begin{equation*}
 \mathcal{F}\left(D\left(M\right)\right) \cong D\left(\mathcal{F}\left(M\right)\right) \cong D\left(M\right).
\end{equation*}
From \cite[2.12]{Lyu97} we know that \F-finite modules only have finitely many associated primes. In contrast to this, in \cite[3.5]{BN08} it is shown that the Matlis duals \\ 
$D(H^i_{(X_1,\dots,X_i)}\left(\Bbbk[[X_1,\dots,X_n]]\right))$ have infinitely many associated primes for $i \le n$. So Matlis duals of \F-finite $R$-modules over complete regular local rings with
perfect residue field are also \F-modules which generally are not \F-finite.
\end{remark}

Now, as a first example, we are going to describe the generalized Matlis dual of the top local cohomology module $H^d_\m\left(R/I\right)$ of a quotient of a complete regular local ring $R$. 
We should keep in mind that the module $H^d_\m\left(R/I\right)$ is just a quasi-$\mathcal{F}$-module, but the generalized Matlis dual $\mathfrak{D}\left(H^d_\m\left(R/I\right)\right)$ will provide us with 
an $\mathcal{F}$-finite module.

\begin{example}\label{exampleD(H)}\cite[4.3.2]{BL01}
Let $\left(R,\m,\Bbbk\right)$ be a complete regular local ring of charactristic $p>0$ and dimension $n$. Let furthermore $I\subseteq R$ be an ideal of $R$ with $height\,I=n-d=c$ and let $S:=R/I$.
Then is $S$ a ring of dimension $d$ and the local cohomology module $H_\m^i\left(S\right)$ is an \F-module when considered as module over $S$.
If we consider $H_\m^i\left(S\right)$ as an $R$-module it is not generally not an \F-module, but only a quasi-\F-module with structure morphism
\begin{equation*}
 \beta\,:\, R^\varphi \otimes_R H^i_\m\left(R/I\right) \rightarrow H^i_\m\left(R/I\right).
\end{equation*}
This map is equivalent to the map induced by the projection 
 $R/I^{[p]} \rightarrow R/I$
under the identification of $R^\varphi \otimes_R H^i_\m\left(R/I\right)$ with $H^i_\m\left(R/I^{[p]}\right)$.
By definition $\mathfrak{D}\left(H^i_\m\left(R/I\right)\right)$ is the limit of the direct system
\begin{equation*}
 D\left(H^i_\m\left(R/I\right)\right) \rightarrow D\left(H^i_\m\left(R/I^{[p]}\right)\right) \rightarrow D\left(H^i_\m\left(R/I^{[p^2]}\right)\right) \rightarrow \hdots 
\end{equation*}
Now we can use the local duality for complete local Gorenstein rings (see \cite[11.29]{IY07}) since $R$ is regular and local. We get an isomorphism of direct systems
\begin{equation*}
\begin{CD}
D\left(H^i_\m\left(R/I\right)\right) @>>> D\left(H^i_\m\left(R/I^{[p]}\right)\right) @>>> D\left(H^i_\m\left(R/I^{[p^2]}\right)\right) @>>> \dots \\
@VV\cong V			@VV\cong V		@VV\cong V		\\
Ext^{n-i}_R\left(R/I,R\right) @>>> Ext^{n-i}_R\left(R/I^{[p]},R\right) @>>> Ext^{n-i}_R\left(R/I^{[p^2]},R\right) @>>> \dots
\end{CD}
\end{equation*}
The maps in the bottom system are induced by the natural projections and thus we have (see i.e. \cite[Remark 7.9]{IY07} for the last isomorphism)
\begin{equation*}
 \mathfrak{D}(H^i_\m(R/I)) = \varinjlim_k D(H^i_\m(R/I^{[p^k]})) \cong \varinjlim_k Ext_R^{n-i}(R/I^{[p^k]},R) \cong \varinjlim_k Ext_R^{n-i}(R/I^k,R).
\end{equation*}
So, all in all we can formulate the following:

\begin{theorem}\label{exampleMatlisAllg}
Let $\left(R,\m,\Bbbk\right)$ be a complete regular local ring of characteristic $p>0$, $\Bbbk$ perfect and $dim\,R=n$. 
Let also $I\subseteq R$ be an ideal of $R$ of $height$ $c=n-d$. Then one has
\begin{equation*}
 \mathfrak{D}\left(H^i_\m\left(R/I\right)\right) \cong H^{n-i}_I\left(R\right)
\end{equation*}
as \F-modules. In particular for the top local cohomology module we have an isomorphism
\begin{equation*}
 \mathfrak{D}\left(H^d_\m\left(R/I\right)\right) \cong H^c_I\left(R\right).
\end{equation*}
\end{theorem}
\begin{proof}
By the characterization of local cohomology as a direct limit of certain $Ext$-modules we get
\begin{equation*}
\varinjlim_k Ext_R^{n-i}\left(R/I^k,R\right) \cong H^{n-i}_I\left(R\right).
\end{equation*}
\end{proof}
\end{example}

\section{Result}

Our aim is now to use the generalized local duality from section \ref{section2} 
to obtain a description of generalized Matlis duals $\mathfrak{D}$ for certain local cohomology quasi-\F-modules, 
which are more general than those from the above example. In the example we used usual local duality to describe modules of the form $\mathfrak{D}\left(H^i_\m\left(R/I\right)\right)$. 
By the results of Hellus we can now, under special assumptions, also describe modules of the form $\mathfrak{D}\left(H^i_{\mathfrak{a}}\left(R/I\right)\right)$.

\begin{theorem}\label{generalizedMatliswithHellus}
Let $\left(R,\m,\Bbbk\right)$ be a complete regular local ring of characteristic $p>0$ with $\Bbbk$ perfect and let $I,\mathfrak{a}\subseteq R$ be ideals of $R$.
Furthermore, let $h\in \N$ be chosen in a way that $H_\mathfrak{a}^l\left(R\right)\neq 0\,\Leftrightarrow l=h$. Then for all $i\in \{0,\dots,h\}$ there is an isomorphism
\begin{equation*}
\mathfrak{D}\left(H_\mathfrak{a}^{h-i}\left(R/I\right)\right)\cong H^i_I\left(D\left(H_\mathfrak{a}^h\left(R\right)\right)\right).
\end{equation*}
\end{theorem}
\begin{proof}
With the given assumptions, we have by Theorem \ref{generalizedMatlis-duality}
\begin{equation*}
 D\left(H_\mathfrak{a}^{h-i}\left(R/I^{[p^k]}\right)\right) \cong Ext_R^i\left(R/I^{[p^k]},D\left(H_\mathfrak{a}^h\left(R\right)\right)\right).
\end{equation*}
Now, by Definition \ref{defgeneralizedMatlis} follows
\begin{align*}
\mathfrak{D}\left(H_\mathfrak{a}^{h-i}\left(R/I\right)\right) &\cong \varinjlim_k D\left(H_\mathfrak{a}^{h-i}\left(R/I^{[p^k]}\right)\right) \\
				       &\cong \varinjlim_k Ext_R^i\left(R/I^{[p^k]},D\left(H_\mathfrak{a}^h\left(R\right)\right)\right) \\
                                       &\cong H_I^i\left(D\left(H_\mathfrak{a}^h\left(R\right)\right)\right).
\end{align*}
\end{proof}

\noindent
Due to Peskine and Szpiro in \cite{PS73}, it is possible to formulate the following.

\begin{lemma}\label{PSIII.4.1}
Let $R$ be a regular domain of characteristic $p>0$ and let $\mathfrak{a} \subseteq R$ be an ideal of $R$, such that $R/\mathfrak{a}$ is Cohen-Macaulay.
Then we obtain
\begin{equation*}
 H^i_\mathfrak{a}\left(R\right)=0 \quad \textrm{für }i \neq height\, \mathfrak{a}.
\end{equation*}
\end{lemma}
\begin{proof}
 \cite[Theorem 21.29]{IY07}.
\end{proof}

\noindent
If we apply this lemma to our result (Theorem \ref{generalizedMatliswithHellus}) we can give the following description of the generalized Matlis duals of certain local cohomology modules.

\begin{theorem}\label{generalizedMatliswithHellus2}
Let $\left(R,\m,\Bbbk\right)$ be a complete regular local ring of characteristic $p>0$ with perfect residue field $\Bbbk$ and let $I,\mathfrak{a}\subseteq R$ be ideals of $R$.
If furthermore $R/\mathfrak{a}$ is Cohen-Macaulay and $i\in \{0,\dots,height\,\mathfrak{a}\}$ is arbitrary, we have 
\begin{equation*}
\mathfrak{D}\left(H_\mathfrak{a}^{height\,\mathfrak{a}-i}\left(R/I\right)\right)\cong H^i_I\left(D\left(H_\mathfrak{a}^{height\,\mathfrak{a}}\left(R\right)\right)\right).
\end{equation*}
\end{theorem}

\begin{proof}
As a regular local ring $R$ is Cohen-Macaulay and a domain (see \cite[11.3; 11.10]{IY07} and \cite[8.18]{IY07}). Hence, we have
\begin{equation*}
 height\,\mathfrak{a} = depth_R\left(\mathfrak{a},R\right).
\end{equation*}
Thus, $H^{height\,\mathfrak{a}}_\mathfrak{a}\left(R\right)\neq 0$ and overall by lemma \ref{PSIII.4.1} we get 
\begin{equation*}
 H^i_\mathfrak{a}\left(R\right) \neq 0 \Longleftrightarrow i=height\,\mathfrak{a}.
\end{equation*}
The claim now follows from Theorem \ref{generalizedMatliswithHellus}.
\end{proof}

\begin{remark}
If we set $\mathfrak{a}$ to be the maximal ideal $\m$ of $R$ and $dim\,R=n$, we have $height\,\m=n$ and $R/\m$ is Cohen-Macaulay. So Theorem \ref{generalizedMatliswithHellus2} yields
\begin{equation*}
\mathfrak{D}\left(H_\m^i\left(R/I\right)\right)\cong H^{n-i}_I\left(D\left(H_\m^n\left(R\right)\right)\right).
\end{equation*}
But since $H_\m^n\left(R\right)$ is artinian, by local duality (\cite[11.29]{IY07}) and Matlis duality (\cite[A.35]{IY07}) we know there are isomorphisms $D\left(H_\m^n\left(R\right)\right)\cong Ext_R^0\left(R,R\right)=R$.
\noindent
And hence, Theorem \ref{generalizedMatliswithHellus2} in fact generalizes Theorem \ref{exampleMatlisAllg}.
\end{remark}

\section{An example}

\noindent
Now we are going to connect our new description of certain generalized Matlis duals to Hartshorne`s example, resp. to the generalization of this example by Stückrad and Hellus 
(see Theorem \ref{GeneralizedExample}).
\\

\noindent
By definition of generalized Matlis duality (\ref{defgeneralizedMatlis}) we have the logical validity of the implication

\begin{equation*}
 H^i_I\left(M\right) \textrm{ artinian quasi-\F-module} \Longrightarrow \mathfrak{D}\left(H^i_I\left(M\right)\right) \textrm{ \F-finite},
\end{equation*}
and hence
\begin{equation*}
\mathfrak{D}\left(H^i_I\left(M\right)\right) \textrm{ not \F-finite} \Longrightarrow H^i_I\left(M\right) \textrm{  not artinian}.
\end{equation*}
\\

\noindent
Therefore we can check the \F-finiteness of $\mathfrak{D}\left(H^i_I\left(M\right)\right)$ to get results about the Artinianess of $H^i_I\left(M\right)$. For the generalization of Hartshorne's example by Stückrad and Hellus 
(Theorem \ref{GeneralizedExample}) we get in particular:

\begin{example}
Let $\Bbbk=\mathbb{F}_p$, $R=\Bbbk[[X_1,\dots ,X_n]]$ the ring of formal powers series in $X_1,\dots ,X_n$ ($n\geq4$) over the finite field $\Bbbk$ and let $\mathfrak{a}$ be the ideal $\left(X_1,\dots ,X_{n-2}\right)R$. So we have
\begin{equation*}
 height\,\mathfrak{a}=n-2
\end{equation*}
and 
\begin{equation*}
 R/\mathfrak{a} \cong R[[X_{n-1}, X_n]]
\end{equation*}
is Cohen-Macaulay. So, by Theorem \ref{generalizedMatliswithHellus2}, we get for $i\in \{0,\dots,height\,\mathfrak{a}\}$ and an ideal $I\subseteq R$
\begin{equation*}
\mathfrak{D}\left(H_\mathfrak{a}^{n-2-i}\left(R/I\right)\right)\cong H^i_I\left(D\left(H_\mathfrak{a}^{n-2}\left(R\right)\right)\right).
\end{equation*}
In particular, if we take $i=0$ and $I=\left(p\right)$ for some prime $p \in \left(X_{n-1},X_n\right)$, we get
\begin{equation*}
\mathfrak{D}\left(H_\mathfrak{a}^{n-2}\left(R/pR\right)\right)\cong H^0_{pR}\left(D\left(H_\mathfrak{a}^{n-2}\left(R\right)\right)\right).
\end{equation*}
\end{example}

\noindent
If we now consider the set of associated primes of the latter module, we see that
\begin{equation*}
 Ass \, (H^0_{pR}(D(H_\mathfrak{a}^{n-2}(R)))) = Ass \, (\Gamma_{pR}(D(H_\mathfrak{a}^{n-2}(R)))) = \{\mathfrak{q} \in Ass \, (D(H_\mathfrak{a}^{n-2}(R))) \, | \, p \in \mathfrak{q} \}.
\end{equation*}
But from \cite[4.3.4]{HE07a}, we know that at least in the case where $p \in \mathfrak{a}$, this set has infinitely many elements and so by \cite[2.12]{Lyu97} the module
\begin{equation*}
 \mathfrak{D}\left(H_\mathfrak{a}^{n-2}\left(R/pR\right)\right)\cong H^0_{pR}\left(D\left(H_\mathfrak{a}^{n-2}\left(R\right)\right)\right)
\end{equation*}
cannot be \F-finite. We have herewith reproven $H_\mathfrak{a}^{n-2}\left(R/pR\right)$ not to be artinian.

\section{Further applications}

As a first consequence of Theorem \ref{generalizedMatliswithHellus2} we can give at least in the regular case a very short proof for \cite[Theorem 7.4.1]{HE07a}.

\begin{corollary}\label{coro1}
Let $\left(R,\m,\Bbbk\right)$ be a complete regular local ring of characteristic $p>0$ with perfect residue field $\Bbbk$ and $x_1, \dots ,x_i$ $\in R$ ($i\geq 1$) a regular sequence in $R$. 
Set $I:=\left(x_1, \dots ,x_i\right)R$. Then we have a natural isomorphism
\begin{equation*}
 H^i_I\left(D\left(H^i_I\left(R\right)\right)\right) \cong E_R\left(\Bbbk\right).
\end{equation*}
\end{corollary}
\begin{proof}
$I$ is a set-theoretic complete intersection ideal of $R$ since it is generated by a regular sequence and so $R/I$ is a complete intersection ring. Therefore, $R/I$ is Cohen-Macaulay 
(\cite[10.5]{IY07}) and by Theorem \ref{generalizedMatliswithHellus2} we have
\begin{equation*}
  H^i_I\left(D\left(H^i_I\left(R\right)\right)\right) \cong \mathfrak{D}\left(H_I^0\left(R/I\right)\right).
\end{equation*}

Further, by the definition of the generalized Matlis duality, we conclude

 \begin{align*}
 \mathfrak{D}\left(H_I^0\left(R/I\right)\right) &= \varinjlim D\left(H_I^0\left(R/I\right)\right) \\
   		    &= \varinjlim Hom_R\left(H_I^0\left(R/I^{[p^k]},E\left(\Bbbk\right)\right)\right) \\
   		    &= \varinjlim Ext^0_R\left(\Gamma_I\left(R/I^{[p^k]},E\left(\Bbbk\right)\right)\right) \\
                      &= \varinjlim Ext^0_R\left(R/I^{[p^k]},E\left(\Bbbk\right)\right)\\
   		    &\cong H_I^0\left(E\left(\Bbbk\right)\right) \\ 	
		    &\cong \Gamma_I\left(E\left(\Bbbk\right)\right) \\ 
 		    &\cong E\left(\Bbbk\right).
 \end{align*}
\end{proof}

\begin{example}
 Let $\Bbbk=\mathbb{F}_p$, $R=\Bbbk[[X_1, \dots ,X_n]]$ the ring of formal powers series in $X_1,\dots ,X_n$ over the finite field $\Bbbk$ and let $I$ be the ideal 
$\left(X_1,\dots ,X_i\right)R$. Then, we have an isomorphism
\begin{equation*}
 H^i_I\left(D\left(H^i_I\left(R\right)\right)\right) \cong E_R\left(\Bbbk\right).
\end{equation*}
\end{example}

As an application of the last corollary and the fact that certain Matlis duals of local cohomology modules are \F-modules (see Remark \ref{remark1}), 
we can extend \cite[Theorem 7.4.2]{HE07a} to the case of a local ring of prime characteristic with perfect coefficent field. 

\begin{theorem}
Let $\left(R,\m,\Bbbk\right)$ be a complete regular local equicharacteristic ring with perfect residue field $\Bbbk$, $I \subseteq R$ an ideal of height $h \geq 1$, 
and assume that
\begin{equation*}
H_I^l\left(R\right)=0 \quad \forall \, l > h \, .
\end{equation*}
Then one has 
\begin{equation*}
 H_I^h\left(D\left(H_I^h\left(R\right)\right)\right)=E_R\left(\Bbbk\right) \quad \mbox{ or } \quad H_I^h\left(D\left(H_I^h\left(R\right)\right)\right)=0.
\end{equation*}
\end{theorem}
\begin{proof}
 If the characteristic of $R$ is zero this is just the statement of \cite[Theorem 7.4.2]{HE07a}. So let $R$ be a complete regular local ring of prime characteristic $p>0$ with perfect residue 
field $\Bbbk$. We can use more or less the same arguments like in the characteristic zero case, but we have to consider the \F-module structure of the Matlis duals (see Remark \ref{remark1})
instead of the structure as a $\mathcal{D}$-module. 
Here are the details. \\ \\
As a regular local ring $R$ is Cohen-Macaulay and hence, we have $height\,I = depth_R\left(I,R\right)$.
So let $x_1, \dots ,x_h$ $\in I$ be an $R$-regular sequence. If we set 
\begin{equation*}
 D:=D\left(H^h_{(x_1,\dots,x_h)R}\left(R\right)\right),
\end{equation*}
from \cite[1.1.4]{HE07a}, we know that $x_1,\dots,x_h$ is also a $D$-regular sequence, so we have $depth\left(\left(x_1,\dots,x_h\right)R,D\right)\geq h$ and hence
\begin{equation}\label{eq1}
 H^i_{\left(x_1,\dots,x_h\right)R}\left(D\right)=0 \qquad \forall \, i<h.
\end{equation}
If we now consider the Grothendieck spectral sequence for composed functors
(see e.g. \cite[10.47]{Rot79})
\begin{equation*}
 E_2^{p,q}=\left(R^pF\right)\left(R^qG\right)A \stackrel{p}{\Longrightarrow} R^n\left(FG\right)A,
\end{equation*}
we can apply this to the composition $\Gamma_I\circ \Gamma_{\left(x_1,\dots,x_h\right)R}$ and the module $D$ and get
\begin{equation*}\
 E_2^{p,q}=\left(H^p_I\right)\left(H^q_{\left(x_1,\dots,x_h\right)R}\right)\left(D\right) \stackrel{p}{\Longrightarrow} H^n\left(\Gamma_I\circ \Gamma_{\left(x_1,\dots,x_h\right)R}\right)\left(D\right)
\end{equation*}
By (\ref{eq1}) this sequence collapses on the $q$-axis, since the only nonzero modules are $E_2^{p,h}$ and hence in consequence of \cite[10.26]{Rot79}
\begin{equation*}
\left(H^0_I\right)\left(H^h_{\left(x_1,\dots,x_h\right)R}\right)\left(D\right) \cong H^h\left(\Gamma_I\circ \Gamma_{\left(x_1,\dots,x_h\right)R}\right)\left(D\right). 
\end{equation*}
As $H^0_I \cong \Gamma_I$ and $\left(x_1,\dots,x_h\right)R \subseteq I$,
\begin{equation}\label{eq2}
 H_I^h\left(D\right) \cong \Gamma_I\left(H^h_{\left(x_1,\dots,x_h\right)R}\left(D\right)\right) \subseteq H^h_{\left(x_1,\dots,x_h\right)R}\left(D\right) \cong E_R\left(\Bbbk\right)
\end{equation}
holds, where the last isomorphism is Corollary \ref{coro1}.\\ \\
 Since $E_R\left(\Bbbk\right)$ is artinian (see \cite[A.32]{IY07}) it is $Supp\,E\left(\Bbbk\right)=\{\m\}$ and hence $dim\, E\left(\Bbbk\right)=0$. 
By (\ref{eq2}) $H_I^h\left(D\right)$ is a submodule of $E_R\left(\Bbbk \right)$ and so we have also $dim\,H_I^h\left(D\right)=0$. 
Because $R$ is an \F-module by Remark \ref{remark1} and \cite[1.2(b)]{Lyu97} $H_I^h\left(D\right)$ is an \F-module and so \cite[1.4]{Lyu97} yields
\begin{equation*}
 inj\,dim\,H_I^h\left(D\right)=0.
\end{equation*}
So $H_I^h\left(D\right)$ is injective and therefore it is isomporph to a direct sum of modules of the form $E_R\left(R/\mathfrak{p}\right)$ for $\mathfrak{p}\in Spec\,R$. 
But as a submodule of $E_R\left(\Bbbk\right)$ it has to be either $E_R\left(\Bbbk\right)$ itself or zero and in addition, the natural injection
\begin{equation*}
H_I^h\left(R\right) \subseteq H^h_{\left(x_1,\dots,x_h\right)R}\left(R\right) 
\end{equation*}
induces a surjection
\begin{equation*}
 D \twoheadrightarrow D\left(H^h_I\left(R\right)\right)
\end{equation*}
and due to the fact that $H_I^h\left(R\right)$ is right-exact, a surjection
\begin{equation*}
 H_I^h\left(D\right) \twoheadrightarrow H_I^h\left(D\left(H_I^h\left(R\right)\right)\right).
\end{equation*}
But again the module on the right side is by Remark \ref{remark1} and \cite[1.2(b);1.4]{Lyu97} an injective \F-module and we have a surjection from zero or $E\left(\Bbbk\right)$ to it. 
Therefore the direct sum decomposition consists at most of one copy of $E\left(\Bbbk\right)$ and we get the asserted statement.

\end{proof}

\end{document}